\documentclass[12pt]{l4dc2023}
\usepackage{times}

\title[Regret Analysis of Online LQR Control]{Regret Analysis of Online LQR Control via Trajectory Prediction and Tracking: Extended Version}

\newtheorem{assumption}{Assumption}
\newtheorem{problem}{Problem}
\newtheorem{prop}[theorem]{Proposition}

\DeclareMathOperator*\uplim{\overline{lim}}
\DeclareMathOperator*{\E}{\mathbf{E}}
\allowdisplaybreaks
\usepackage{comment}
\usepackage{todonotes}
\DeclareMathOperator*{\argmin}{argmin}

\author{%
 \Name{Yitian Chen$^*$} \Email{yitian.chen@anu.edu.au}\\
 \Name{Timothy L.\ Molloy$^*$} \Email{timothy.molloy@anu.edu.au}\\
 \Name{Tyler Summers$^\dag$} \Email{tyler.summers@utdallas.edu}\\
 \Name{Iman Shames$^*$} \Email{iman.shames@anu.edu.au}\\
 \addr $^*$CIICADA Lab, The Australian National University \quad\quad $^\dag$The University of Texas at Dallas%
}

\begin{document}

\maketitle

\begin{abstract}%
 In this paper, we propose and analyze a new method for online linear quadratic regulator (LQR) control with a priori unknown time-varying cost matrices. The cost matrices are revealed sequentially with the potential for future values to be previewed over a short window. Our novel method involves using the available cost matrices to predict the optimal trajectory, and a tracking controller to drive the system towards it. We adopted the notion of dynamic regret to measure the performance of this proposed  online LQR control method, with our main result being that the (dynamic) regret of our method is upper bounded by a constant. Moreover, the regret upper bound decays exponentially with the preview window length, and is extendable to systems with disturbances. We show in simulations that our proposed method offers improved performance compared to other previously proposed online LQR methods.
\end{abstract}


\section{Introduction}
\indent Optimal control problems arise in many fields such as econometrics \citep{bjork_time-inconsistent_2021,radneantu_making_2009}, robotics \citep{hampsey_exploiting_2022,renganathan_towards_2020}, physics \citep{liu_nonlinear_2021} and machine learning \citep{westenbroek_feedback_2020}. The Linear Quadratic Regulator (LQR) problem is the archetypal optimal control problem with vector-valued states and controls, and is reviewed in the following. Consider a controllable linear time-invariant system
\begin{equation}\label{eq:linsys}
    x_{t+1} = Ax_{t} + Bu_{t} + w_{t},
\end{equation}
where $t$ is a nonegative integer, $m$ and $n$ are positive integers, $A \in \mathbb{R}^{n\times n}$, $B \in \mathbb{R}^{n\times m}$, $x_t,w_{t}\in\mathbb{R}^n$, and $x_{0} =\bar{x}_{0}$ for some given $\bar{x}_0 \in \mathbb{R}^{n}$, and $u_{t} \in \mathbb{R}^{m}$. 
For a given finite time horizon $T \geq 2$ and initial condition $\bar{x}_{0}$, the control decisions $\{u_{t}\}_{t=0}^{T-2}$ are computed to minimize the quadratic cost function
\begin{align}
    J_{T}(\{x_{t}\}_{t=0}^{T-1},\{u_{t}\}_{t=0}^{T-2}) := \sum_{t=0}^{T-2} x_{t}^{\mathsf{T}}Q_{t}x_{t} + u_{t}^{\mathsf{T}}R_{t}u_{t} + x_{T-1}^{\mathsf{T}}Q_{T-1}x_{T-1}, \label{eq:cost}
\end{align}
where $Q_{t} \in \mathbb{S}^{n}_{+}$ and $R_{t} \in \mathbb{S}^{m}_{++}$ are time-varying cost matrices and $\mathbb{S}^{n}_{+}$ and $\mathbb{S}^{n}_{++}$ denote the sets of positive semi-definite symmetric and positive definite symmetric matrices, respectively. 
The states $x_{t}$ and controls $u_{t}$ minimizing \eqref{eq:cost} must satisfy \eqref{eq:linsys}.
When the cost matrices $\{Q_t\}_{t=0}^{T-1}$ and $\{R_t\}_{t = 0}^{T-2}$ are known \emph{a priori}, the controls minimizing \eqref{eq:cost} subject to \eqref{eq:linsys} can be found in closed form, cf.\ \citep[Chapter 2]{anderson_optimal_2007}.
However, in many real word applications, such as power systems \citep{kouro_model_2009}, chemistry \citep{chen_distributed_2012} and mechatronics \citep{vukov_real-time_2015}, full information about the cost matrices over the whole time horizon is not available (in advance) to the decision maker.\\
\indent In our work, for a given time horizon $T$ and preview window length $0 \leq W \leq T-2$, we suppose that at any time $t$ where $0 \leq t < T-2-W$, only the initial condition of the system \eqref{eq:linsys} and the (partial) sequences of cost matrices $\{Q_{i}\}_{i=0}^{t+W}$ and $\{R_{i}\}_{i=0}^{t+W}$ are known.
Let the cost-function information available to the decision maker at time $t$ be
\begin{equation}\label{eq:history}
    \mathcal{H}_{t} := \{\{Q_{i}\}_{i=0}^{t+W}, \{R_{i}\}_{i=0}^{t+W}, \bar{x}_{0}\},
\end{equation}
where $\mathcal{H}_{t}$ contains the full temporal information about the cost matrices for $t \geq T-2-W$. 
The main focus of our work is to propose a novel control policy that generates $u_{t}$ using the information available at time $t$, and investigate its performance.
We specifically consider a feedback control policy $\pi(\cdot,\cdot)$ of the form 
\begin{equation}\label{eq:policyForm}
    u_{t} = \pi(x_{t}, \mathcal{H}_{t}),
\end{equation}
and adopt the notion of regret to measure its performance.
 Several different notions of regret have been well studied and explored in the online optimization problem, including static regret \citep{zinkevich_online_2003,shalev-shwartz_online_2012}, dynamic regret \citep{jadbabaie_online_2015}. In our work, performance is measured by dynamic regret. For any control sequence $\{u_{t}\}_{t=0}^{T-2}$ and associated state sequence $\{x_{t}\}_{t=0}^{T-1}$, the dynamic regret is defined as
\begin{equation}\label{eq:regret}
    \text{Regret}_{T}(\{u_{t}\}_{t=0}^{T-2}) 
    := J_{T}(\{x_{t}\}_{t=0}^{T-1},\{u_{t}\}_{t=0}^{T-2}) - J_{T}(\{x_{t}^{*}\}_{t=0}^{T-1},\{u_{t}^{*}\}_{t=0}^{T-2}),
\end{equation}
where
\begin{equation}\label{eq:contseq}
    \{u_{t}^{*}\}_{t=0}^{T-2} := \argmin_{\{\upsilon_{i}\}_{i=0}^{T-2}} J_{T}(\{\xi_{i}\}_{i=0}^{T-2},\{\upsilon_{i}\}_{i=0}^{T-2}),
\end{equation}
and $\{x_{t}^{*}\}_{t=0}^{T-1}$ satisfy the system dynamics \eqref{eq:linsys} for input sequence $\{u_{t}^{*}\}_{t=0}^{T-2}$.

\subsection{Related Works}
Similar investigations have recently been conducted in \cite{cohen_online_2018}, \cite{zhang_regret_2021}, and \cite{akbari_logarithmic_2022}.
\cite{cohen_online_2018} and \cite{akbari_logarithmic_2022} consider a different notion of regret involving comparison with controls $\tilde{u}_t = -K\tilde{x}_t$ (instead of $u_t^*$) generated by a fixed gain $K$ from the set of $(\bar{\kappa}, \bar{\gamma})$ strongly stable gains denoted by $\mathcal{K}$. More precisely, $\mathcal{K}$ is the set of all gains where for any $K \in \mathcal{K}$, there exists matrices $L$ and $H$ such that $A+BK = HLH^{-1}$, with $\left \| L \right \| \leq 1-\bar{\gamma}$ and $\left \| H \right \|, \left \| H^{-1} \right \| \leq \bar{\kappa}$ for prescribed scalars $\bar{\kappa}$ and $\bar{\gamma}$\footnote{We shall use $\|\cdot\|$ to denote either the 2-norm of a vector or the spectral norm of a matrix, depending on its argument.}. For a sequence of controls $\{u_{t}\}_{t=0}^{T-1}$, the notion of regret for time horizon $T$ and controls $\{u_{t}\}_{t=0}^{T-1}$ from these works is
\begin{align}
    \text{StablizingRegret}_{T}(\{u_{t}\}_{t=0}^{T-2}) :=  J_{T}(\{x_{t}\}_{t=0}^{T-1},\{u_{t}\}_{t=0}^{T-2}) - J_{T}(\{\tilde{x}_{t}\}_{t=0}^{T-1},\{\tilde{K}\tilde{x}_{t}\}_{t=0}^{T-2}), \label{eq:oldreg}
\end{align}
where $\tilde{K} \in \argmin_{K\in \mathcal{K}} J_{T}(\{\tilde{x}_{t}\}_{t=0}^{T-1},\{K\tilde{x}_{t}\}_{t=0}^{T-2})$ and $\{\tilde{x}_{t}\}_{t=0}^{T-1}$ satisfies \eqref{eq:linsys}.

\cite{cohen_online_2018} propose an online LQR algorithm that yields controls with a theoretical regret upper bound of $\text{StablizingRegret}_{T}(\{u_{t}\}_{t=0}^{T-1})\leq O(\sqrt{T})$. 
However, the algorithm involves a computationally expensive projection step at each time $t$, and the projection set can become empty for some controllable systems when the covariance of the system disturbances $w_t$ is positive definite\footnote{For example, the set is empty, if $A=
    \begin{pmatrix}
    1 & 2\\
    6 & 9
    \end{pmatrix}$, 
    $B= \begin{pmatrix}
        9\\
        6
    \end{pmatrix}$, and the disturbances are distributed according to a multivariate Gaussian with mean zero and covariance matrix $I_2$.}.  Thus, this method is not applicable to all controllable linear time-invariant systems. Moreover, the theoretical stabilizing regret upper bound is proportional to the inverse of the cube of lower bound of covariance of system disturbances, i.e., $\text{StablizingRegret}_{T}({u_t}_{t=0}^{T-1}) = O(\frac{1}{\sigma^3})$, where the covariance of disturbances from $\eqref{eq:linsys}$ is lower bounded by $\sigma^2 I$. If $\sigma = 0$, the theoretical regret upper bound is undefined.
    \cite{akbari_logarithmic_2022} proposed an Online Riccati Update algorithm that obtains $\text{{StablizingRegret}}_{T}(\{u_{t}\}_{t=0}^{T-1}) = O(\sigma^2\log(T))$. The result avoids the undefined regret upper bound of \cite{cohen_online_2018} when the covariance matrix is not lower bounded by a positive $\sigma$. However, like \cite{cohen_online_2018}, the performance of the algorithm proposed in \cite{akbari_logarithmic_2022} is only guaranteed to achieve sublinear \emph{stabilizing regret} \eqref{eq:oldreg} against the best \emph{fixed} control gain $K$ from the set $\mathcal{K}$.
This notion of regret is not suitable for dynamic non-stationary environments. For example, a self-driving car may operate in different environments such as high-wind areas, or high and low-friction road surfaces. For the best performance to counter-act these environments, we need to use time-varying control gains and compare them against the best time-varying policies chosen in hindsight. 
 
\cite{zhang_regret_2021} investigate the dynamic regret \eqref{eq:regret} offered by an online LQR approach inspired by model predictive control. Future cost matrices and predicted disturbances are assumed to be available over a short future preview window of length $W \geq 0$, and the following assumption is made.
\begin{assumption}\label{assumption:cost}
    There exist symmetric positive definite matrices $Q_{min}, Q_{max}, R_{min}, R_{max}$ such that for time $0 \leq t \leq T-2$, 
\begin{equation}\label{eq:costbounds}
    \begin{split}
    0 \prec Q_{min} \preceq Q_{t} \preceq Q_{max},\\
        0 \prec R_{min} \preceq R_{t} \preceq R_{max},
    \end{split}
\end{equation}
where $F \prec G$ denotes $G - F$ being positive definite for symmetric matrices $F$ and $G$.
\end{assumption}

Under Assumption \ref{assumption:cost}, \cite{zhang_regret_2021} propose an online LQR algorithm for selecting controls $u_{t}$ at time $t$ by solving
\begin{align*}
    \min_{\{u_{k}\}_{k=t}^{t+W}} \sum_{k=t}^{t+W} x_{k}^{\mathsf{T}}Q_{k}x_{k} + u_{k}^{\mathsf{T}}R_{k}u_{k} + x_{t+W+1}^{\mathsf{T}}P_{max}x_{t+W+1}
\end{align*}
subject to \eqref{eq:linsys} where $P_{max}$ is the solution of the algebraic Riccati equation for the infinite-horizon LQR problem with cost matrices $Q_{max}$ and $R_{max}$.
The dynamic regret \eqref{eq:regret} of control sequences generated by this method is shown to be upper bound by a quantity that shrinks exponentially as the preview window length increases. 
However, the estimate of the tail cost at each time step (i.e., $x_{t+W+1}^{\mathsf{T}}P_{max}x_{t+W+1}$) can be too pessimistic due to its reliance on $P_{max}$ and the matrices $Q_{max}$ and $R_{max}$ from the bounds given in Assumption \ref{assumption:cost}. 

\subsection{Contributions}
The key contributions of this paper are:
\begin{itemize}
    \item The proposal of a method for solving the online LQR problem that is independent of the given upper or lower bounds on the cost matrices;
    \item Development of a regret bound for the disturbance-free case and proof that our proposed control policy yields sublinear regret;
    \item Provision of sufficient conditions under which our regret bound is less than that of the state-of-the-art methodology; and
    \item Analysis of our regret bound in the presence of disturbances.
\end{itemize}


\paragraph*{\textbf{Outline.}} The rest of the paper is organized as follows. In Section \ref{section:problemFormulation}, we state the online LQR problem that we consider.
In Section \ref{section:approach}, we introduce our proposed online LQR algorithm and bound its dynamic regret. In Section \ref{section:numerics}, we provide numerical results for the simulation of our proposed algorithm. Concluding remarks are presented in the last section.

\section{Problem Formulation}\label{section:problemFormulation}
In this paper, we consider the following problem.


\begin{problem}[Online LQR]\label{problem:OnlineLQR}  Consider the controllable system \eqref{eq:linsys}. Let the cost matrices in \eqref{eq:regret} satisfy Assumption \ref{assumption:cost} for any given $T \geq 2$ and $W < T-2$. At time $0 \leq t \leq T-W-2$, the available information to the decision maker is given by $\mathcal{H}_{t}$ as defined in \eqref{eq:history}. It is desired to design a control policy $\pi(\cdot, \cdot)$ of the form \eqref{eq:policyForm} that yields a regret, as defined by \eqref{eq:regret}, that is independent of the bounds given in Assumption \ref{assumption:cost}. Moreover, we seek to establish appropriate regret bounds for the following cases:
\begin{itemize}
    \item[a)] The case where $w_{t} = 0$ for $0\leq t\leq T-1$;
    \item[b)] The case where the disturbances $w_t$ for $0\leq t\leq T-1$ are independent and identically distribution (i.i.d.) random variables such that $\E(w_{t}) = 0$ and $\E(w_{t}w_{t}^{\mathsf{T}}) = W_{d}$ with $\E(\cdot)$ being the expectation operator and $W_{d} \in \mathbb{S}^{n}_{+}$. 
\end{itemize}

\end{problem}
Specifically, for part a) of Problem~\ref{problem:OnlineLQR} we show that the regret (as defined in \eqref{eq:regret}) associated with our proposed control policy is sublinear with respect to the time horizon $T$ for the case where $w_{t} = 0$ for $0\leq t\leq T-1$, i.e., 
\begin{equation}
    \text{Regret}_{T}(\{u_{t}\}_{t=0}^{T-2}) \leq o(T).
\end{equation}
For part b), we define the notion of ``expected regret'' as
\begin{align}\label{eq:expectedregret_def}
        \text{ExpectedRegret}_{T}(\{u_{t}\}_{t=0}^{T-2}) := \E(J_{T}(\{x_{t}\}_{t=0}^{T-1},\{u_{t}\}_{t=0}^{T-2}) - J_{T}(\{x_{t}^{*}\}_{t=0}^{T-1},\{u_{t}^{*}\}_{t=0}^{T-2})),
    \end{align}
and show that our proposed control policy yields controls that satisfy
$$ \text{ExpectedRegret}_{T}(\{u_{t}\}_{t=0}^{T-2}) \leq C_{ER} T \gamma^{2W} $$ for positive scalars $C_{ER}$ and $\gamma$\footnote{The exact definition of $\gamma$ will be presented in Theorem~\ref{theorem:mainResult}.}. In what follows we address this problem.

\section{Approach and Regret Analysis}
\label{section:approach}
Our proposed online LQR approach involves first using the available information $\mathcal{H}_{t}$ at each time $t$ to predict the optimal state $x_{t+1}^{*}$ solving the full information LQR problem described in \eqref{eq:contseq}. We then select controls to track this prediction.
At time $0 \leq t \leq T-1$, we only know the information in $\mathcal{H}_{t}$. Let $x_{t+1\mid t+W}$ denote the estimate of the optimal state at time $t+1$ based on $\mathcal{H}_{t}$. We aim to track to the state $x_{t+1\mid t+W}$ at time $t+1$.
\paragraph{Prediction.} 
At each time $t$, we plan an optimal trajectory starting from the initial state $\bar{x}_{0}$ using the known cost matrices up to time $t+W$ and setting all the future matrices to be equal to their known values for time $t+W$.
Specifically, at time $t$ where $0\leq t < T-W$, define $J_{t+W}(\cdot,\cdot)$ as
\begin{align}
J_{t+W}(\{\xi_{i}\}_{i=0}^{T-1},\{\upsilon_{i}\}_{i=0}^{T-2}) &:= \sum_{k=0}^{t+W} [\xi_{k}^{\mathsf{T}}Q_{k}\xi_{k} + \upsilon_{k}^{\mathsf{T}}R_{k}\upsilon_{k}]\notag \\& + \sum_{k=t+1+W}^{T-2}[\xi_{k}^{\mathsf{T}}Q_{t+W}\xi_{k} + \upsilon_{k}^{\mathsf{T}}R_{t+W}\upsilon_{k}] + \xi_{T-1}^{\mathsf{T}}Q_{t+W}\xi_{T-1},
\end{align}
and 
\begin{align}
J_{t+W}(\{\xi_{i}\}_{i=0}^{T-1},\{\upsilon_{i}\}_{i=0}^{T-2}) := J_{T}(\{\xi_{i}\}_{i=0}^{T-1},\{\upsilon_{i}\}_{i=0}^{T-2})
\end{align}
for $T-W \leq t \leq T-1$.

Then, we find the predicted optimal control sequence for all $0\leq j\leq T-2$ by solving
\begin{equation}\label{eq:seq}
    \begin{split}
        \left ( \{x_{j\mid t+W}\}_{j=0}^{T-1},\{u_{j\mid t+W}\}_{j=0}^{T-2} \right ) = \argmin_{\left( \{\xi_{i}\}_{i=0}^{T-1},\{\upsilon_{i}\}_{i=0}^{T-2} \right )} & \quad J_{t+W}(\{\xi_{i}\}_{i=0}^{T-1},\{\upsilon_{i}\}_{i=0}^{T-2})\\
    \text{subject to}\quad&\quad  \xi_{i+1} = A\xi_{i}+B\upsilon_{i},\quad \xi_{0} = \bar{x}_{0}.
    \end{split}
\end{equation}
\paragraph{Prediction Tracking.}
We propose the following feedback control policy
\begin{align}\label{eq:policy}
    \pi(x_{t},\mathcal{H}_{t}) = K(x_{t}-x_{t\mid t+W}) + u_{t\mid t+W},
\end{align}
where $K\in \mathbb{R}^{m\times n}$ is a control matrix such that $\rho(A+BK) < 1$, and $\rho(\cdot)$ denotes the matrix spectral radius. Intuitively, such control matrix $K$ leads to contraction of the distance 
between $x_{t+1}$ and $x_{t+1\mid t+W}$, respectively given by (\ref{eq:linsys}) and (\ref{eq:seq}).
\subsection{Regret Analysis for the Disturbance-free Case}
In the following theorem, we present the result for the case of Problem \ref{problem:OnlineLQR}a) that the control sequence generated by (\ref{eq:policy}) incurs a sublinear upper bound regret with respect to time horizon $T$. Here, with slight abuse of notation, for a sequence of matrices $\{\Sigma_{i}\}_{i=0}^{N}$, we define $\max_{0 \leq t \leq N} \Sigma_{t} := \{\Sigma_{\tau} \mid 0\leq \tau \leq N, \Sigma_{\tau}\succeq \Sigma_{k}\text{ for all }0 \leq k \leq N\}$ and $\min_{0 \leq t \leq N} \Sigma_{t} := \{\Sigma_{\tau} \mid 0\leq \tau \leq N, \Sigma_{\tau}\preceq \Sigma_{k}\text{ for all }0 \leq k \leq N\}$. This enables us to define \emph{cost matrix sequence extrema} as $\bar{R}_{max} := \max_{0 \leq t \leq T-2}R_{t}$, $\bar{Q}_{max} := \max_{0 \leq t \leq T-1}Q_{t}$, $\bar{R}_{min} := \min_{0 \leq t \leq T-2}R_{t}$, and $\bar{Q}_{min} := \min_{0 \leq t \leq T-1}Q_{t}$. For any matrix $\Gamma$, we further define $\lambda_{min}(\Gamma)$ as the minimum eigenvalue of $\Gamma$ and $\lambda_{max}(\Gamma)$ as the maximum eigenvalue of $\Gamma$. 
\begin{theorem}[Main Result] \label{theorem:mainResult} Consider the linear system defined by \eqref{eq:linsys}. For a given time horizon $T \geq 2$ and preview window length $0 \leq W \leq T-2$. Suppose that at time $0 \leq t \leq T-2$ the control input $u_t$ is generated by policy $\pi(\cdot,\cdot)$ as given by (\ref{eq:policy}).
Under Assumption \ref{assumption:cost}, the regret defined by \eqref{eq:regret} satisfies
\begin{equation}\label{eq:regUpperbound}
    \begin{split}
        \text{Regret}_{T}(\{u_{t}\}_{t=0}^{T-2}) 
    &\leq \frac{10D\gamma^{2W}\|\bar{x}_{0}\|^2}{3}\bigg[(\alpha_{1}+\alpha_{2})(\frac{C^2C_{K}\gamma}{(\gamma-1)})^2\bigg(\gamma^2S_{T}(\eta^2\gamma^2)-2\gamma S_{T}(\eta^2\gamma)\\
    &+S_{T}(\eta^2))+\frac{10C_f^{2}}{3}((\frac{\eta\gamma}{q(q-\eta\gamma)}-\frac{\eta}{q(q-\eta)})^2S_{T}(q^2)\\
    &+\frac{(\eta\gamma)^2S_{T}(\eta^2\gamma^2)}{q^2(q-\eta\gamma)^2}+\frac{\eta^2S_{T}(\eta^2)}{q^2(q-\eta)^2})\bigg)+(C_KC^2)^2S_{T}(\eta^2)\bigg],
    \end{split}
\end{equation}
where $\bar{P}_{max}$ satisfies
$$\bar{P}_{max} = \bar{Q}_{max} + A^{\mathsf{T}}\bar{P}_{max}A - A^{\mathsf{T}}\bar{P}_{max}B(\bar{R}_{max}+B^{\mathsf{T}}\bar{P}_{max}B)^{-1}B^{\mathsf{T}}\bar{P}_{max}A,$$ $D = \left\|\bar{R}_{max} + B^{\mathsf{T}}\bar{P}_{max}B\right\|$,
$C_{K} = \left\|(\bar{R}_{min}+B^{\mathsf{T}}\bar{Q}_{min}B)^{-1}\right\|^{2}\left\|\bar{R}_{max}B^{\mathsf{T}}\right\|\frac{\lambda_{max}^{2}(\bar{P}_{max})}{\lambda_{min}(\bar{Q}_{min})},$
$C = \frac{\lambda_{max}(\bar{P}_{max})}{\lambda_{min}(\bar{Q}_{min})}\label{eq:C}$,
$\eta = \sqrt{1-\frac{\lambda_{min}(\bar{Q}_{min})}{\lambda_{max}(\bar{P}_{max})}}$,
$\alpha = \max_{\substack{0 \leq i \leq t-1\\ 0\leq t \leq T-2}}\{\lambda_{max}(A^{\mathsf{T}}P_{i+1}^{*}A), \lambda_{max}(A^{\mathsf{T}}P_{i+1\mid t}A)\}$,
$\beta = \min_{0\leq t \leq T-2}\lambda_{min}(Q_{t})$,
$\gamma = \frac{\alpha}{\alpha+\beta}\label{eq:gamma}$,
$S_{T}(z) = \sum_{t=0}^{T-1} z^{t}$,
$\alpha_{1}=\max_{t} \left\|K_{t\mid t+W}-K\right\|^2$,
$\alpha_{2}=\max_{t} 2\left\|K_{t}^{*}-K\right\|^2$,
$C_f = \max_{n \geq 0} \frac{\left\|(A+BK)^{n}\right\|}{(q+\varepsilon)^n}$,
$q = \rho(A+BK) + \varepsilon$,
and $0 \leq \varepsilon < 1-\rho(A+BK)$.
\end{theorem}
\begin{proof}
    See Appendix \ref{sec:proof_main}.
\end{proof}
\begin{remark}
For any $z\in[0,1)$ there exists an $\Lambda\in\mathbb{R}$, such that $\lim_{T\rightarrow \infty} S_T(z) = \Lambda$
. Consequently,
    $\uplim_{T\rightarrow \infty}
    \frac{\text{Regret}_{T}(\{u_{t}\}_{t=0}^{T-2})}{T}
    = 0$, which implies that the control sequence described by (\ref{eq:policy}) yields sublinear regret.
\end{remark}
\begin{remark}
    Let $F(\bar{x}_{0},A,B,T,\bar{R}_{max},\bar{R}_{min},\bar{Q}_{max},\bar{Q}_{min},K)$ denote the right hand side (RHS) of  \eqref{eq:regUpperbound}. By stating almost identical lemmas to Lemmas \ref{thm:lemma2} and \ref{thm:lemma3} using the bounds given in Assumption \ref{assumption:cost} instead of the cost matrices sequence extrema values, one can arrive at a regret bound in terms of these bounds analogous to \eqref{eq:regUpperbound}:  
    \begin{equation*}
        \text{Regret}_{T}(\{u_{t}\}_{t=0}^{T-2}) \leq F(\bar{x}_{0},A,B,T,R_{max},R_{min},Q_{max},Q_{min},K).
    \end{equation*}
\end{remark}
In the following proposition, we state a condition in terms of the bounds given in Assumption \ref{assumption:cost} and the cost matrices sequence extrema where it is guaranteed that the bound given in the above theorem is smaller than that of \cite[Theorem 1, Equation (15)]{zhang_regret_2021}. Obviously, there might be other conditions, exploration of which is left to future work.
\begin{prop}\label{thm:sufficient}
    Adopt the hypothesis of Theorem \ref{theorem:mainResult}. If
    \begin{align}\label{eq:sufficientUpperbound}
        \lambda_{max}^{10}(Q_{max}) \geq \frac{5\left[(1+\frac{\alpha_{1}+\alpha_{2}}{(1-\gamma)^2})(\frac{1}{1-\eta^2})+\frac{10C_{f}^{2}}{q^{2}(q-\eta\gamma)^{2}(q-\eta)^{2}(1-\eta^{2})(1-\eta^{2}\gamma^{2})(1-q^2)}\right]}{(C_{K}^{2}\lambda^{2}_{min}(\bar{R}_{min})\lambda^{4}_{min}(\bar{Q}_{min}))^{-1}6\|A\|^2\|B\|^2\|B\bar{R}^{-1}_{min}B^{\mathsf{T}}\|^2},
    \end{align}
    where $Q_{max}$ is given in Assumption \ref{assumption:cost},
    then the RHS of inequality in \cite[Theorem 1, Equation (15)]{zhang_regret_2021} is greater than the RHS of inequality \eqref{eq:regret} in Theorem \ref{theorem:mainResult}.
\end{prop}
\begin{proof}
    See Appendix \ref{section:proofSufficient}.
\end{proof}
The RHS of \eqref{eq:sufficientUpperbound} is independent of the matrices $Q_{min},Q_{max},R_{min},R_{max}$ given in Assumption \ref{assumption:cost}. On the other hand, the upper bound of regret for control decisions that generated by \cite[Algorithm 1]{zhang_regret_2021} does depend on these values and even if the actual sequence of the cost matrices remain bounded away from these bounds, the method still explicitly uses the bounds and this is a potential source of conservatism.
\subsection{Regret Analysis in the Presence of Disturbances}\label{section:disturbance}
The result presented in the following theorem address Problem \ref{problem:OnlineLQR} case b). Note that at time $t$, $\{w_{k}\}_{k=0}^{t}$ is the available sequence of disturbances to the decision maker. In this case, we still consider a policy $\pi(\cdot,\cdot)$ as given by (\ref{eq:policy}) with the only difference that $x_{t\mid t+W}$ is obtained by solving the following optimisation problem:
\begin{equation}\label{eq:seq_w}
    \begin{split}
        \left ( \{x_{j\mid t+W}\}_{j=0}^{T-1},\{u_{j\mid t+W}\}_{j=0}^{T-2} \right ) = \argmin_{\left( \{\xi_{i}\}_{i=0}^{T-1},\{\upsilon_{i}\}_{i=0}^{T-2} \right )} & \quad J_{t+W}(\{\xi_{i}\}_{i=0}^{T-1},\{\upsilon_{i}\}_{i=0}^{T-2})\\
    \text{subject to}\quad    &\quad \xi_{i+1} = A\xi_{i} + B\upsilon_{i} + w_{i} \quad \text{for $0 \leq i \leq t$},\\
    &\quad \xi_{i+1} = A\xi_{i} + B\upsilon_{i} \quad \text{for $i > t$} ,\quad \xi_{0} = \bar{x}_{0}.
    \end{split}
\end{equation}

\begin{theorem}\label{theorem:disturbanceResult}
    Consider the system defined by \eqref{eq:linsys}. For a given time horizon $T \geq 2$ and preview window length $0 \leq W \leq T-2$. Suppose that at time $0 \leq t \leq T-2$ the control input $u_t$ is generated by policy $\pi(\cdot,\cdot)$ as given by (\ref{eq:policy}).
Under Assumption \ref{assumption:cost}, the expected regret defined by \eqref{eq:expectedregret_def} satisfies
\begin{equation}
    \text{ExpectedRegret}_{T}(\{u_{t}\}_{t=0}^{T-2}) \leq C_{ER}T\gamma^{2W}
\end{equation}
where $C_{ER}$ is a positive scalar and $\gamma$ is given in Theorem \ref{theorem:mainResult}.
\end{theorem}
\begin{proof}
    See Appendix \ref{sec:proof_disturbance}.
\end{proof}
In the next section, we investigate the performance of the proposed algorithm for different scenarios.
\section{Numerical Simulations}\label{section:numerics}
In this section, we numerically demonstrate the performance of the proposed algorithm. To this end, define $\Phi_{T,W} := \text{Regret}_{T}(\{u^{'}_{t}\}_{t=0}^{T-2}) - \text{Regret}_{T}(\{u_{t}\}_{t=0}^{T-2})$, where $\{u^{'}_{t}\}_{t=0}^{T-2}$ is generated from \cite[Algorithm 1]{zhang_regret_2021} and $\{u_{t}\}_{t=0}^{T-2}$ is generated by the policy described in \eqref{eq:policy}, under preview window length of $W$.
\subsection{Linearized Inverted Pendulum} \label{section:Linearized_Inverted_Pendulum} 
Consider the following linearized inverted pendulum system \cite[Chapter 2.13]{franklin_feedback_2020}:  
\begin{align}\label{eq:invertedPendulum}
    x_{t+1} =
    \begin{pmatrix}
    0 & 1 & 0 & 0\\
    0 & -0.1818 & 2.6727 & 0\\
    0 & 0 & 0 & 1\\
    0 & -18.1818 & 31.1818 & 0
    \end{pmatrix}
    x_{t}
    +
    \begin{pmatrix}
    0\\
    1.8182\\
    0\\
    4.5455
    \end{pmatrix}
    u_{t}.
\end{align}
In the following experiments, the preview horizon $W$ ranges from 0 to 19 and the time horizon $T$ ranges from 19 to 500. The cost matrices are chosen uniformly satisfying by Assumption \ref{assumption:cost} with $Q_{min} = 8\times10^{3}I_{4\times 4}$, $Q_{max} = 3.2\times10^{4}I_{4\times 4}$, $R_{min} = 2\times10^{3}$, and $R_{max} = 9.8\times10^{4}$. The fixed controller from \eqref{eq:policy} is chosen by placing the poles at the location of $(1,6,4,3)\times 10^{-3}$. We repeat the experiment in 200 trials. Figure \ref{fig:physSysReg} demonstrate $\Phi_{T,W}$, under preview window length from 0 to 19 and time horizon from 19 to 500. As the preview window length greater than 2, our method outperforms \cite[Algorithm 1]{zhang_regret_2021}.
\begin{figure}
\caption{Performance measure $\Phi_{T,W}$ for simulated systems.}
        \label{fig:different_cases}
     \centering
     \subfigure[$\Phi_{T,W}$ for disturbance-free linearized inverted pendulum system]{
         \includegraphics[width=.45\textwidth]{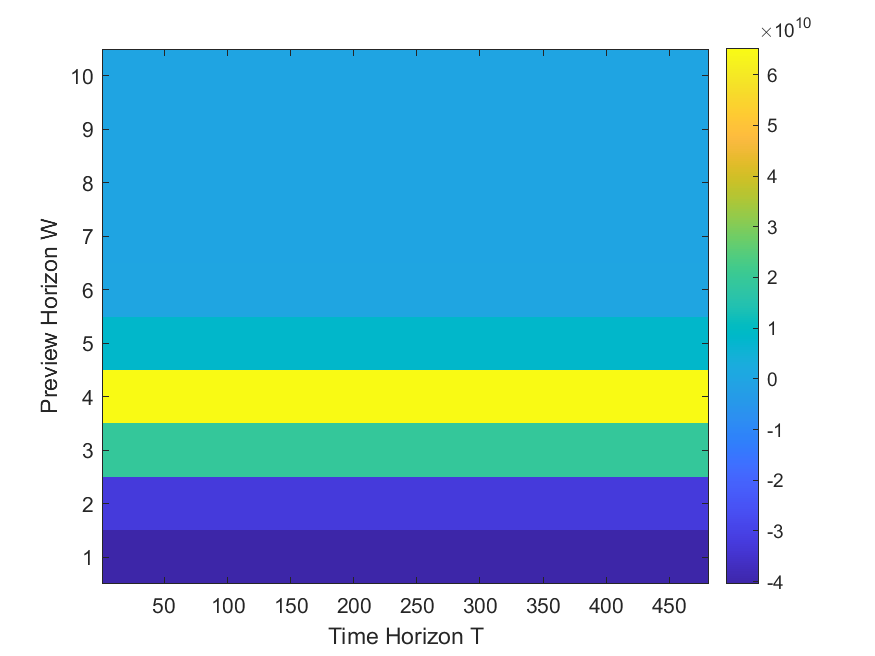}
    \label{fig:physSysReg}}
     \subfigure[$\Phi_{T,W}$ for disturbance-free random controllable systems]{
         \includegraphics[width=.45\textwidth]{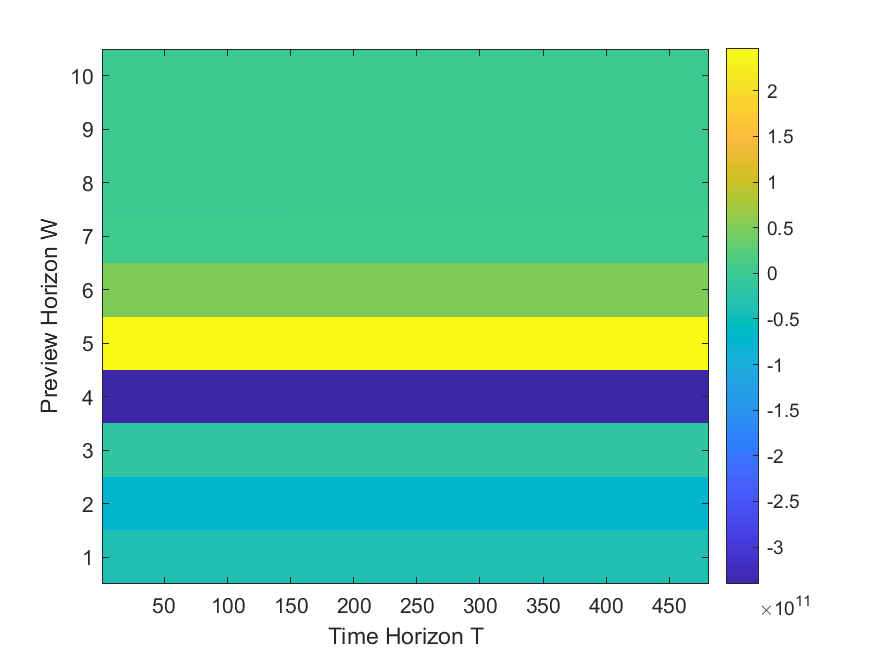}
    \label{fig:ranSysReg}}
    \\
     \subfigure[$\Phi_{T,W}$ for linearized inverted pendulum system with disturbances]{
         \includegraphics[width=.45\textwidth]{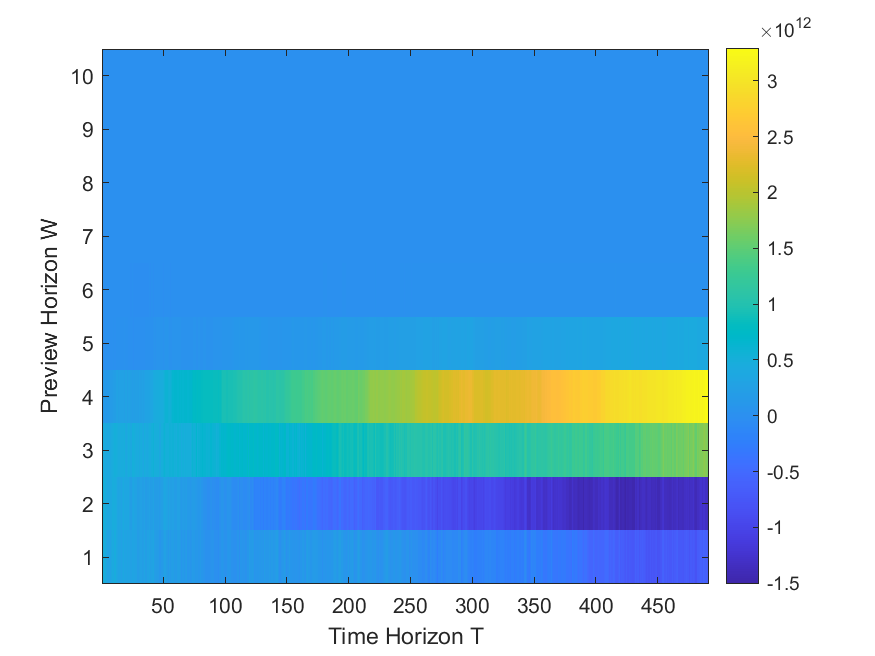}
    \label{fig:PhysSysDisReg}}
     \subfigure[$\Phi_{T,W}$ for random controllable system with disturbances]{
         \includegraphics[width=.45\textwidth]{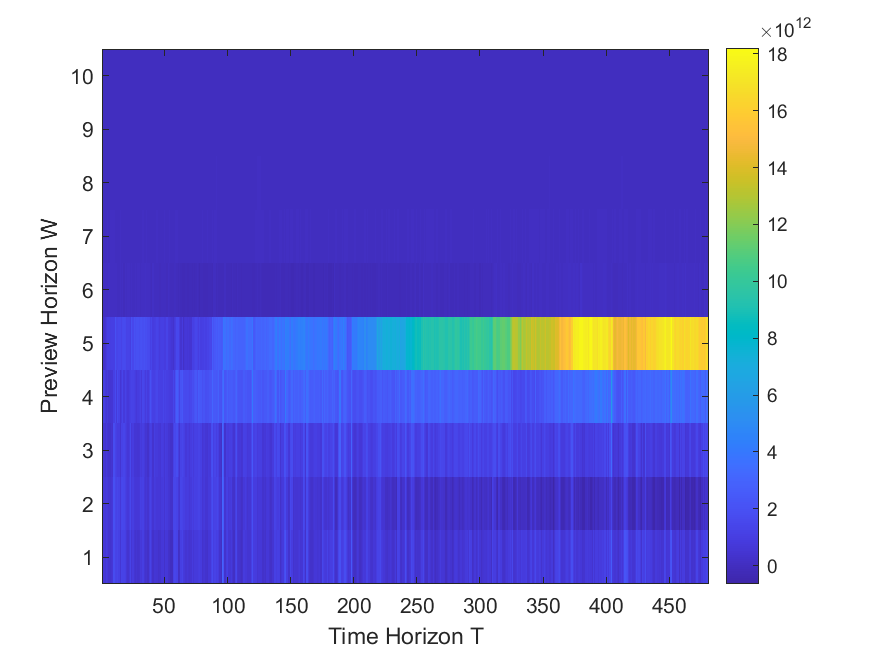}
    \label{fig:RanSysDisReg}}
\end{figure}
\subsection{Random Linear Systems}\label{section:Random_Linear_Systems}
In this experiment, the linear system is randomly chosen where all elements of $A$ and $B$ are drawn uniformly within the range of $(0,10)$ and ensure the pairs of $(A,B)$ are controllable. The setting of preview window length, time horizon, cost matrices and the pole location for the control matrix $K$ from \eqref{eq:policy} are identical as what we have chosen in Section \ref{section:Linearized_Inverted_Pendulum}. The plot in Figure \ref{fig:ranSysReg} demonstrates the subtraction between the regret of control decision generated by \cite[Algorithm 1]{zhang_regret_2021} and the regret of control decision generated by our proposed method, by averaging the regret over 200 trials. As the preview window length exceed 4, our method outperforms \cite[Algorithm 1]{zhang_regret_2021}.

The plots from Figure \ref{fig:physSysReg} and \ref{fig:ranSysReg} demonstrate that, as the preview window length exceeds the rank of the system, which is the least number of steps that require to steer the state of the system to a designated state, the proposed method outperforms the method from \cite{zhang_regret_2021}.
\subsection{Linear Systems with disturbances}
The following experiments repeat the ones considered in Section \ref{section:Linearized_Inverted_Pendulum} and \ref{section:Random_Linear_Systems}, using the system defined in \eqref{eq:invertedPendulum} and in the presence of disturbance $w_{t} \sim \mathcal{N}(0,25I_{4\times 4})$. The setting of the preview window length, time horizon, cost matrices and the pole location for the control matrix $K$ are identical as what we have chosen from the experiment in Section \ref{section:Linearized_Inverted_Pendulum}. The method of finding $x_{t\mid t+W}$ and $u_{t\mid t+W}$ can be referred to Remark \ref{section:disturbance}. The plots in Figures \ref{fig:PhysSysDisReg} and \ref{fig:RanSysDisReg} depict the average value of $\Phi_{T,W}$ after 200 random trials. 

\section{Conclusions and Future Work}\label{section:conclusions}
This paper propose a new control policy that achieves constant dynamic regret where the available information of cost matrices are sequentially reviewed as time step increases. The proposed method and consequently its regret have been demonstrated to be, contrary to the state-of-the-art, independent of the \emph{ex ante} upper and lower bound of the cost matrices. To exhibit the effect such independence, a sufficient condition is provided under which the regret upper bound of the proposed method is guaranteed to be smaller that the one from \cite[Theorem 1]{zhang_regret_2021}. This paper leads to many interesting research direction which are briefly discussed below. It would be interesting to devise a methodology for selecting a time-varying feedback gain matrix in in \eqref{eq:policy} instead of a fixed $K$ in order to further minimise the regret. Moreover, one can extend the algorithm to the case of time-varying $A_{t}$ and $B_{t}$ for the system matrices and via differential dynamic programming for nonlinear dynamics with control constraints, and establish new dynamic regret.







\bibliography{research0}

\appendix
\section{Proof of Theorem~\ref{theorem:mainResult}}\label{sec:proof_main}
Before stating the proof of the theorem we introduce the necessary propositions and lemmas.
\begin{prop}\label{thm:prop1}
    Let
    \begin{align*}
        Q_{i\mid t+W} := 
        \begin{cases}
            Q_{i} & \text{if $i \leq t+W$}\\
            Q_{t+W} & \text{if $i > t+W$},
        \end{cases}\\
        R_{i\mid t+W} := 
        \begin{cases}
            R_{i} & \text{if $i \leq t+W$}\\
            R_{t+W} & \text{if $i > t+W$}.
        \end{cases}
    \end{align*}
    For $0\leq t\leq T-1$, $0\leq i\leq T-2$ the analytical solution of control sequence described at \eqref{eq:seq} is given by
    \begin{align*}
        P_{T-1\mid t+W} &= Q_{t\mid t+W},\\
        P_{i\mid t+W} &= A^{\mathsf{T}}P_{i+1\mid t}A + Q_{i\mid t}+A^{\mathsf{T}}P_{i+1\mid t}BK_{i\mid t},\\
        K_{i\mid t+W} &:= -(R_{i\mid t+W}+B^{\mathsf{T}}P_{i+1\mid t+W}B)^{-1}B^{\mathsf{T}}P_{i+1\mid t+W}A,\\
        u_{i\mid t+W} &= K_{i\mid t+W}x_{i\mid t+W},\\
        x_{i+1\mid t+W} &= Ax_{i\mid t+W} + Bu_{i\mid t+W}.
    \end{align*}
    The control sequence that minimize the cost in \eqref{eq:regret} can be found by setting $t = T-1$.
\end{prop}
The above proposition is a consequence of \cite[Chapter 2.4]{anderson_optimal_2007}.

The next lemma states that the matrices $P_{i\mid t}$ described in the above proposition are upper and lower bounded if the cost matrices $Q_{i}$ and $R_{i}$ are upper and lower bounded.
\begin{lemma}\label{thm:lemmaPmax}
For $0 \leq t,i\leq T-1$, under Assumption \ref{assumption:cost}, there exists a positive definite matrix $\bar{P}_{max}$ such that $
    Q_{min} \preceq \bar{Q}_{min} \preceq P_{i\mid t} \preceq \bar{P}_{max}.$
\end{lemma}
The proof of the above lemma is similar to the proof in 
 \cite[Appendix D, Proposition 11]{zhang_regret_2021}. Based on the previous proposition and lemma, the following lemmas reveals the upper bound of matrix norm for $P_{i\mid t_{0}}-P_{i\mid t}$ and $K_{i\mid t}-K_{i\mid t_{0}}$ for any $0 \leq i \leq t\leq t_0\leq T-1$. These upper bounds can infer the exponential stability of control matrices $K_{i\mid t}$ and $K_{i\mid t_{0}}$ described in Proposition \ref{thm:prop1}.

\begin{lemma}\label{thm:lemma2}
For any $0 \leq i \leq t\leq t_{0}\leq T-1$, the following hold:
\begin{align*}
    \left\|P_{i\mid t}-P_{i\mid t_{0}}\right\| \leq \frac{\lambda_{max}^{2}(\bar{P}_{max})}{\lambda_{min}(\bar{Q}_{min})}\gamma^{t+1-i},
\end{align*}
and
\begin{align*}
    \left\|K_{i\mid t}-K_{i\mid t_{0}}\right\| \leq C_{K}\gamma^{t-i}.
\end{align*}

\end{lemma}
\begin{proof}
Before proceeding with the proof, we first define $\delta_{\infty}(\cdot,\cdot)$ as
\begin{align*}
    \delta_{\infty}(X,Y) := \|\log(X^{-\frac{1}{2}}YX^{\frac{1}{2}})\|,
\end{align*}
for positive semi-definite matrices $X$ and $Y$.
Since $Q_{i},R_{i},P_{i\mid t_{0}},P_{i\mid t}$ are positive definite, by \cite[Lemma D.2]{krauth_finite-time_2019}, we have that
\begin{align*}
    \delta_{\infty}(P_{i\mid t},P_{i\mid t_0})
    &= \delta_{\infty}(Q_{t} + A^{\mathsf{T}}P_{i+1\mid t_0}^{\frac{1}{2}}(I+P_{i+1\mid t_0}^{\frac{1}{2}}BR_{t}^{-1}B^{\mathsf{T}}P_{i+1\mid t_0}^{\frac{1}{2}})^{-1}P_{i+1\mid t_0}^{\frac{1}{2}}A,\\
    &\quad Q_{t} + A^{\mathsf{T}}P_{i+1\mid t}^{\frac{1}{2}}(I+P_{i+1\mid t}^{\frac{1}{2}}BR_{t}^{-1}B^{\mathsf{T}}P_{i+1\mid t}^{\frac{1}{2}})^{-1}P_{i+1\mid t}^{\frac{1}{2}}A) \\
    &\leq \frac{\alpha}{\alpha+\beta}\delta_{\infty}((P_{i+1\mid t_0}^{-1}+BR_{t}^{-1}B^{\mathsf{T}})^{-1},(P_{i+1\mid t}^{-1}+BR_{t}^{-1}B^{\mathsf{T}})^{-1})
    \\
    &\leq \frac{\alpha}{\alpha+\beta}\delta_{\infty}(P_{i+1\mid t_0},P_{i+1\mid t})\leq \gamma\delta_{\infty}(P_{i+1\mid t_0},P_{i+1\mid t}).
\end{align*}
Furthermore, we have
\begin{align*}
    \delta_{\infty}(P_{i\mid t_0},P_{i\mid t}) \leq \gamma^{t-i+1}\delta_{\infty}(P_{t+1\mid t_0},P_{t+1\mid t}).
\end{align*}
Based on \cite[Lemma 6]{zhang_regret_2021}, we can further deduce the last step above to
\begin{align*}
    \delta_{\infty}(P_{i\mid t_0},P_{i\mid t}) \leq \underbrace{\gamma^{t-i+1}\log(\frac{\lambda_{max}(\bar{P}_{max})}{\lambda_{min}(\bar{Q}_{min})}).}_{c:=}
\end{align*}
Based on \cite[Lemma 7]{zhang_regret_2021}, we can conclude that
\begin{align*}
    \|P_{i\mid t_0}-P_{i\mid t}\| &\leq \lambda_{max}(\bar{P}_{max})\frac{e^{c}-1}{c}\delta_{\infty}(P_{i\mid t_0},P_{i\mid t})\\
    &\leq \lambda_{max}(\bar{P}_{max})\frac{e^{c}-1}{c}c\\
    &\leq \lambda_{max}(\bar{P}_{max})(e^{c}-1) \\
    &\leq \gamma^{t-i+1}\frac{\lambda^{2}_{max}(\bar{P}_{max})}{\lambda_{min}(\bar{Q}_{min})}.
\end{align*}
The last step of inequality above is a consequence of  the inequality  $e^{ax} - 1 \leq ae^{x}$
for $0 \leq a \leq 1$ and $x \geq 0$.  Furthermore, since
\begin{align*}
    \|K_{i\mid t}-K_{i\mid t_0}\|
    &= \|-(R_{i}+B^{\mathsf{T}}P_{i+1\mid t}B)^{-1}B^{\mathsf{T}}P_{i+1\mid t}A+(R_{i}+B^{\mathsf{T}}P_{i+1\mid t_0}B)^{-1}B^{\mathsf{T}}P_{i+1\mid t_0}A\|.
\end{align*}
Let $G_{1} = (R_{i}+B^{\mathsf{T}}P_{i+1\mid t}B)$, $G_{2}=(R_{i}+B^{\mathsf{T}}P_{i+1\mid t_0}B)$. We rearrange the above equations as
\begin{align}
        \|\!K_{i\mid t}\!-\!K_{i\mid t_0}\!\|\!
    &=\! \|G_{1}^{-1}B^{\mathsf{T}}P_{i+1\mid t}-G_{2}^{-1}B^{\mathsf{T}}P_{i+1\mid t_0}\| \notag\\
    &=\! \|G_{1}^{-1}G_{2}^{-1}G_{2}B^{\mathsf{T}}P_{i+1\mid t}-G_{2}^{-1}G_{1}^{-1}G_{1}B^{\mathsf{T}}P_{i+1\mid t_0}\| \notag\\
    &\leq\! \|G_{2}^{-1}G_{1}^{-1}\|\|(G_{1}G_{2})(G_{1}^{-1}G_{2}^{-1})G_{2}B^{\mathsf{T}}P_{i+1\mid t}-G_{1}B^{\mathsf{T}}P_{i+1\mid t_0}\| \notag\\
    &=\! \|G_{2}^{-1}G_{1}^{-1}\|\|((G_{1}G_{2})(G_{1}^{-1}G_{2}^{-1})\!-\!I)G_{2}B^{\mathsf{T}}P_{i+1\mid t}+G_{2}B^{\mathsf{T}}P_{i+1\mid t}-G_{1}B^{\mathsf{T}}P_{i+1\mid t_0}\| \notag\\
    &\leq\! \|G_{2}^{-1}G_{1}^{-1}\|(\|(G_{1}G_{2}-G_{2}G_{1})G_{1}^{-1}B^{\mathsf{T}}P_{i+1\mid t}\|+\|G_{2}B^{\mathsf{T}}P_{i+1\mid t}-G_{1}B^{\mathsf{T}}P_{i+1\mid t_0}\|)\notag\\
    &\leq\! \|G_{2}^{-1}G_{1}^{-1}\|\|G_{2}B^{\mathsf{T}}P_{i+1\mid t}-G_{1}B^{\mathsf{T}}P_{i+1\mid t_0}\|. \label{eq:lemmaPIneq1}
\end{align}
The last two steps from the above is due to $G_{1}$, $G_{2}$ being symmetric matrices, being a $G_{1}G_{2}-G_{2}G_{1}$ is skew symmetric matrix, and the fact that the induced $2$-norm of a skew symmetric matrix is 0. 
Moreover,
\begin{align}
    \left \|(R_{i}+B^{\mathsf{T}}\right. & \left. P_{i+1\mid t_0}B) B^{\mathsf{T}}P_{i+1\mid t}  -(R_{i}+B^{\mathsf{T}}P_{i+1\mid t}B)B^{\mathsf{T}}P_{i+1\mid t_0}\right \|\notag\\ 
    &\leq \|R_{i}B^{\mathsf{T}}(P_{i+1\mid t} - P_{i+1\mid t_0})\| + \|B\|\|P_{i+1\mid t_0}(BB^{\mathsf{T}})P_{i+1\mid t}-P_{i+1\mid t}(BB^{\mathsf{T}})P_{i+1\mid t_0}\| \notag\\
    &\leq \|R_{i}B^{\mathsf{T}}\|\|P_{i+1\mid t}-P_{i+1\mid t_0}\|.\label{eq:lemmaPIneq2}
\end{align}
Thus, substituting \eqref{eq:lemmaPIneq2} in \eqref{eq:lemmaPIneq1}, we have
\begin{align*}
    \|K_{i\mid t}-K_{i\mid t_0}\|
    &\leq \|G_{2}^{-1}G_{1}^{-1}\|\|\bar{R}_{max}B^{\mathsf{T}}\|\|P_{i+1\mid t}-P_{i+1\mid t_0}\|\\
    &\leq \|G_{2}^{-1}G_{1}^{-1}\|\|\bar{R}_{max}B^{\mathsf{T}}\|\frac{\lambda_{max}^{2}(\bar{P}_{max})}{\lambda_{min}(\bar{Q}_{min})}\gamma^{t-i} \\
    &\leq \|(\bar{R}_{min}+B^{\mathsf{T}}\bar{Q}_{min}B)^{-1}\|^{2}\|\bar{R}_{max}B^{\mathsf{T}}\| \frac{\lambda_{max}^{2}(\bar{P}_{max})}{\lambda_{min}(\bar{Q}_{min})}\gamma^{t-i}= C_{K}\gamma^{t-i}.
\end{align*}
\end{proof}
\begin{lemma}\label{thm:lemma3}
For time horizon $T$, suppose $0 \leq t_{0} \leq t_{1} \leq t \leq T-2$. Suppose $C$ and $\eta$ are given in Theorem \ref{theorem:mainResult}.
Then,
\begin{equation*}
    \left \|\prod_{i = t_{0}}^{t_{1}} (A+BK_{i\mid t})\right \| \leq C\eta^{t_{1}-t_{0}+1}.
\end{equation*}
\end{lemma}
This lemma can be proved following the same steps as those found in the proof of \cite[Appendix E,Proposition 2]{zhang_regret_2021}.

The next lemma establishes the bound of the distance between the state that generated by control law at (\ref{eq:seq}) and the optimal state that generated by (\ref{eq:contseq}), together with the the differences between the control defined at (\ref{eq:contseq}) and (\ref{eq:seq}).
\begin{lemma}\label{thm:lemma4}
For $t \geq 1$, the norm of the difference between the state vector $x_{t}$ generated by \eqref{eq:policy} and the optimal state vector $x_{t}^{*}$ generated by control sequence \eqref{eq:contseq} is
\begin{equation}\label{eq:stateDifferenceBound}
    \|x_{t}-x_{t}^*\| \leq \frac{C^2C_K\|\bar{x}_{0}\|\gamma^W}{\gamma-1}(\eta^{t-1}\gamma(\gamma^t-1)
    +C_f(\frac{\eta\gamma}{q}(\frac{q^{t-1}-(\eta\gamma)^{t-1}}{q-\eta\gamma})-\frac{\eta}{q}(\frac{q^{t-1}-\eta^{t-1}}{q-\eta})).
\end{equation}
Moreover, the norm of the difference between the predicted trajectory and the optimal trajectory at time $t$ is given by
\begin{align*}
    \left \|x_{t\mid t+W}-x_{t}^{*}\right \| \leq \frac{C^{2}C_{K}\|\bar{x}_{0}\|\gamma^W}{\gamma-1}\eta^{t-1}\gamma(\gamma^{t}-1).
\end{align*}
\end{lemma}
\begin{proof}\label{section:lemma4_proof}
    Observe the dynamics of $x_{t} - x_{t}^{*}$, we have
\begin{align*}
    x_{t}-x_{t}^{*}= x_{t} - x_{t\mid T-1}= x_{t} - x_{t\mid t+W} + x_{t\mid t+W} - x_{t\mid T-1}.
\end{align*}
Define $\omega_{t} := x_{t}-x_{t\mid t+W}$, $\theta_{i\mid p,q} := x_{i\mid p}-x_{i\mid q}$, where $i \leq p\leq q\leq T$. Consequently, $\omega_{0} = 0$, and
\begin{align*}
    \omega_{t+1} 
    &= (A+BK)x_{t} + B(K_{t\mid t+W}-K)x_{t\mid t+W} \\
    &\quad- (A+BK_{t\mid t+W})x_{t\mid t+W} + x_{t+1\mid t+W} - x_{t+1\mid t+1+W}\\
    &= (A+BK)x_{t} + B(K_{t\mid t+W}-K)x_{t\mid t+W}\\
    &\quad- (A+BK+B(K_{t\mid t+W}-K))x_{t\mid t+W} + x_{t+1\mid t+W} - x_{t+1\mid t+W+1}\\
    &= (A+BK)(x_{t}-x_{t\mid t+W}) + x_{t+1\mid t+W} - x_{t+1\mid t+1+W}\\
    &= (A+BK)\omega_{t} + \theta_{t+1\mid t+W,t+1+W}\\
    &= \sum_{j=1}^{t+1}(A+BK)^{t+1-j}\theta_{j\mid j-1+W,j+W}.
\end{align*}
We now investigate the dynamics of $\theta_{i\mid p,q}$. Note that $\theta_{0\mid p,q} = 0$, and
\begin{align*}
    \theta_{i+1\mid p,q} &= x_{i+1\mid p} - x_{i+1\mid q}\\
    &= (A+BK_{i\mid p})x_{i\mid p} - (A+BK_{i\mid q})x_{i\mid q}\\
    &= (A+BK_{i\mid p})(\theta_{i\mid p,q}+x_{i\mid q}) - (A+BK_{i\mid q})x_{i\mid q}\\
    &= (A+BK_{i\mid p})\theta_{i\mid p,q}+B(K_{i\mid p}-K_{i\mid q})x_{i\mid q}.
\end{align*}
This implies that
\begin{align*}
    x_{i+1\mid p}-x_{i+1\mid q}
    &= \sum_{n=0}^{i}\bigg(\prod_{m=n+1}^{i}(A+BK_{m\mid p})\bigg)B(K_{n\mid p}-K_{n\mid q})\\
    &\qquad\quad\bigg(\prod_{m=0}^{n-1}(A+BK_{m\mid q})\bigg)\bar{x}_{0}.
\end{align*}
By Lemma \ref{thm:lemma3}, we can bound the product term by
\begin{align*}
    &\|\prod_{m=n+1}^{i}(A+BK_{m\mid p})\| \leq C\eta^{i-n},\\
    &\|\prod_{m=0}^{n-1}(A+BK_{m\mid q})\| \leq C\eta^{n}.
\end{align*}
By Lemma \ref{thm:lemma2}, we have
\begin{align*}
    \|B(K_{n\mid p} - K_{n\mid q})\| \leq C_{K}\gamma^{p-n}.
\end{align*}
Thus,
\begin{align*}
    \|\theta_{i+1\mid p,q}\| &= \|x_{i+1\mid p}-x_{i+1\mid q}\| \leq C^{2}C_{K}\sum_{t=0}^{i}\eta^{i}\gamma^{p-t}\\&=C^{2}C_{K}\gamma^{p}\eta^{i}\sum_{t=0}^{i}\frac{1}{\gamma^{t}}= \frac{C^{2}C_{K}\gamma^{p}\eta^{i}}{1-\frac{1}{\gamma}}(1-(\frac{1}{\gamma})^{i+1}).
\end{align*}
Choosing $i = t$, $p=t+W$, and $q=T-1$, results in
\begin{equation}\label{eq:estimateStateDifferenceBound}
\begin{split}
    \|\theta_{t\mid t+W,T-1}\|
    &\leq \frac{C^{2}C_{K}\|\bar{x}_{0}\|\gamma^{t+W}\eta^{t-1}}{1-(\frac{1}{\gamma})}(1-(\frac{1}{\gamma})^{t})\\
    &= \frac{C^{2}C_{K}\|\bar{x}_{0}\|\eta^{t-1}\gamma^{W+1}}{\gamma-1}(\gamma^{t}-1).
\end{split}
\end{equation}
Moreover,
\begin{align*}
    \|\theta_{i\mid i-1+W,i+W}\| &\leq \frac{C^{2}C_{K}\|\bar{x}_{0}\|\gamma^{i-1+W}\eta^{i-1}}{1-(\frac{1}{\gamma})}(1-(\frac{1}{\gamma})^{i})\\
    &= \frac{C^{2}C_{K}\|\bar{x}_{0}\|\eta^{i-1}\gamma^{W}}{\gamma-1}(\gamma^{i}-1).
\end{align*}
Define $\mu_{i,t} = \|(A+BK)^{t-i}\|$. Conclude the above, we have
\begin{equation}\label{eq:realEstStateBound}
    \begin{split}
        \|x_{t} - x_{t\mid t+W}\| &\leq \sum_{i=1}^{t}\|(A+BK)^{t-i}\theta_{i\mid i-1+W,i+W}\|\\
    &\leq \frac{C^2C_K\|\bar{x}_{0}\|\gamma^{W}}{\gamma-1}\sum_{i=1}^{t}\mu_{i,t}\eta^{i}(\gamma^i-1).
    \end{split}
\end{equation}
Thus, 
\begin{align}\label{eq:stateDifferenceBound}
    \|x_{t}-x_{t}^{*}\|
    &\leq \frac{C^2C_K\|\bar{x}_{0}\|\gamma^{W}}{\gamma-1}\bigg(\sum_{i=1}^{t}\mu_{i,t}\eta^{i}(\gamma^i-1)+ \eta^{t-1}\gamma(\gamma^{t}-1) \bigg).
\end{align}
Let $\rho(A+BK) < 1$, for any given $1-\rho(A+BK) > \varepsilon > 0$, by Gelfand's formula, there exist a $T' > 0$, such that for $n > T'$, we have
\begin{align*}
    | \|(A+BK)^{n}\|^{\frac{1}{n}} - \rho(A+BK)| < \varepsilon.
\end{align*}
Thus, $\|(A+BK)^{n}\|^{\frac{1}{n}} < \rho(A+BK) + \varepsilon$ and that implies that
\begin{align*}
    \|(A+BK)^{n}\| < (\rho(A+BK) + \varepsilon)^{n}.
\end{align*}
Define
\begin{align*}
    &C_f := \max_{n \geq 0} \frac{\|(A+BK)^{n}\|}{(q+\varepsilon)^n},\\
    &q := \rho(A+BK) + \varepsilon,
\end{align*}
we can conclude that, for every $r \geq s > 0$ and $\rho(A+BK)< 1$, there exist a pair of $C_f$ and $0 < q < 1$ such that
\begin{align*}
    \mu_{s,r} < \|(A+BK)^{r-s}\| < C_{f}q^{r-s}.
\end{align*}
Thus, the upper bound of the difference between the state vector and the optimal state vector is given by
\begin{align*}
    \|x_{t}-x_{t}^*\| &\leq \frac{C^2C_K\|\bar{x}_{0}\|\gamma^W}{\gamma-1}(\eta^{t-1}\gamma(\gamma^t-1)\\
    &+C_f(\frac{\eta\gamma}{q}(\frac{q^{t-1}-(\eta\gamma)^{t-1}}{q-\eta\gamma})-\frac{\eta}{q}(\frac{q^{t-1}-\eta^{t-1}}{q-\eta})).
\end{align*}
\end{proof}

\begin{lemma}[\cite{zhang_regret_2021}]\label{thm:lemma5}
The regret defined by \eqref{eq:regret} can be written as
\begin{align*}
    \text{Regret}_{T}(\{u_{t}\}_{t=0}^{T-2}) = \sum_{t=0}^{T-1} (u_{t}-\bar{u}_{t})^{\mathsf{T}}(R_{t}+B^{\mathsf{T}}P_{t+1}^{*}B)(u_{t}-\bar{u}_{t}),
\end{align*}
where $\bar{u}_{t} = K_{t}^{*}x_{t}$.
\end{lemma}
We also need the following elementary result.
\begin{lemma}\label{thm:lemmaineq}
    For any $a_1, a_2, a_3 \in \mathbb{R}$, we have that
\begin{align}\label{eq:sqineq}
    (a_1+a_2+a_3)^2 \leq \frac{10}{3}(a_1^{2}+a_{2}^{2}+a_3^2).
\end{align}
\end{lemma}
\begin{proof}
    Note that,
$
    (a_1+a_2+a_3)^2 \leq 2(a_1+a_2)^{2} + 2a_3^2 \leq 4a_1^2 + 4a_2^2 + 2a_3^2$.
Similarly, $    (a_1+a_2+a_3)^2 \leq 4a_1^2 + 2a_2^2 + 4a_3^2$ and $
    (a_1+a_2+a_3)^2 \leq 2a_1^2 + 4a_2^2 + 4a_3^2.$
    
Combining all the above inequalities, yields
\begin{equation*}
    (a_1+a_2+a_3)^2 \leq \frac{4+4+2}{3}(a_1^{2}+a_{2}^{2}+a_3^2)= \frac{10}{3}(a_1^{2}+a_{2}^{2}+a_3^2).
\end{equation*}
\end{proof}
The proof of theorem is given below.
\paragraph{Proof of Theorem~\ref{theorem:mainResult}.}
    Note that
    \begin{equation}\label{eq:contUpperbound}
    \begin{aligned}
        \|u_{t}-\bar{u}_{t}\|^{2}
    &= \|Kx_{t} + (K_{t\mid t+W}-K)x_{t\mid t+W} - K_{t}^{*}x_{t}\|^{2}\\
    &= \|(K_{t\mid t+W}-K)(x_{t\mid t+W}-x_{t}^{*}) - (K_{t}^{*}-K)(x_{t}-x_{t}^{*}) + (K_{t}^{*}-K_{t\mid t+W})x_{t}^{*}\|^{2}\\
    &\leq \frac{10}{3}(\|(K_{t\mid t+W}-K\|^{2}\|x_{t\mid t+W}-x_{t}^{*}\|^{2} + \|K_{t}^{*}-K\|^{2}\|x_{t}-x_{t}^{*}\|^{2} + C_{K}^{2}\gamma^{2W}\|x_{t}^{*}\|^{2}).
    \end{aligned}
\end{equation}
The last inequality above is obtained from Lemma \ref{thm:lemmaineq}. By Lemma \ref{thm:lemma4}, we have that
\begin{equation}\label{eq:stateUpperbound1}
\begin{aligned}
    \sum_{t=0}^{T-1}\|x_{t}-x_{t}^*\|^2 
    &\leq \sum_{t=1}^{T-1} \bigg(\frac{C^2C_K\|\bar{x}_{0}\|\gamma^W}{\gamma-1}(\eta^{t-1}\gamma(\gamma^t-1)\\
    &\quad +C_f(\frac{\eta\gamma}{q}(\frac{q^{t-1}-(\eta\gamma)^{t-1}}{q-\eta\gamma})-\frac{\eta}{q}(\frac{q^{t-1}-\eta^{t-1}}{q-\eta}))\bigg)^2\\
    &\leq 2(\frac{C^2C_{K}\|\bar{x}_{0}\|\gamma^{W+1}}{(\gamma-1)})^2\bigg(\gamma^2S_{T}(\eta^2\gamma^2)-2\gamma S_{T}(\eta^2\gamma)\\
    &\quad +S_{T}(\eta^2))+\frac{10C_f^{2}}{3}((\frac{\eta\gamma}{q(q-\eta\gamma)}-\frac{\eta}{q(q-\eta)})^2S_{T}(q^2)\\
    &\quad +\frac{(\eta\gamma)^2S_{T}(\eta^2\gamma^2)}{q^2(q-\eta\gamma)^2}+\frac{\eta^2S_{T}(\eta^2)}{q^2(q-\eta)^2})\bigg),
\end{aligned}
\end{equation}
and
\begin{equation}\label{eq:stateUpperbound2}
\begin{split}
    \sum_{t=0}^{T-1}\|x_{t\mid t+W}-x_{t}^*\|^2 
    \leq (\frac{C^2C_{K}\|\bar{x}_{0}\|\gamma^{W+1}}{(\gamma-1)})^2\bigg(\gamma^2S_{T}(\eta^2\gamma^2)-2\gamma S_{T}(\eta^2\gamma)
    +S_{T}(\eta^2)\bigg).
\end{split}
\end{equation}
By Lemma \ref{thm:lemma3}, we have
\begin{equation}\label{eq:stateUpperbound3}
\begin{split}
    \sum_{t=0}^{T-1} \|x_{t}^{*}\|^{2} 
    = \sum_{t=0}^{T-1} \|\prod_{i=0}^{t}(A+BK_{i}^{*})\bar{x}_{0}\|^{2}
    \leq (C^{2}\|\bar{x}_{0}\|)^{2}S_{T}(\eta^{2}).
\end{split}
\end{equation}
Substitute \eqref{eq:stateUpperbound1}, \eqref{eq:stateUpperbound2} and \eqref{eq:stateUpperbound3} in \eqref{eq:contUpperbound}, by Lemma \ref{thm:lemma5}, the $\text{Regret}_{T}(\{u_{t}\}_{t=0}^{T-2})$ can be upper bounded by
\begin{align*}
    \text{Regret}_{T}(\{u_{t}\}_{t=0}^{T-2}) 
    &\leq \frac{10D}{3}\bigg(\sum_{t=1}^{T-1}(\frac{C^2C_{K}\|\bar{x}_{0}\|\gamma^{W+1}}{\eta(\gamma-1)}\eta^{t}(\gamma^{t}-1)\|K_{t\mid t+W}-K\|)^2 \\
    &\quad +\|K_{t}^{*}-K\|^2\|x_{t}-x_{t}^{*}\|^2 + (C_{K}C^2\|\bar{x}_{0}\|\gamma^W)^2S_{T}(\eta^2)\bigg)\\
    &\leq \frac{10D\gamma^{2W}\|\bar{x}_{0}\|^2}{3}\bigg[(\alpha_{1}+\alpha_{2})(\frac{C^2C_{K}\gamma}{(\gamma-1)})^2\bigg(\gamma^2S_{T}(\eta^2\gamma^2)-2\gamma S_{T}(\eta^2\gamma)\\
    &\quad +S_{T}(\eta^2))+\frac{10C_f^{2}}{3}((\frac{\eta\gamma}{q(q-\eta\gamma)}-\frac{\eta}{q(q-\eta)})^2S_{T}(q^2)\\
    &\quad +\frac{(\eta\gamma)^2S_{T}(\eta^2\gamma^2)}{q^2(q-\eta\gamma)^2}+\frac{\eta^2S_{T}(\eta^2)}{q^2(q-\eta)^2})\bigg)+(C_KC^2)^2S_{T}(\eta^2)\bigg].
\end{align*}
\hfill $\blacksquare$
\section{Proof of Proposition \ref{thm:sufficient}}\label{section:proofSufficient}
For the sake of clarity of presentation we drop the arguments of $F$ and we use $F'$ to represent the RHS of \cite[Theorem 1, Equation (15)]{zhang_regret_2021} in what follows. 
Note that
\begin{align*}
    F^{'}&\geq \frac{4\|\bar{x}_{0}\|^2D\|A\|^2\|B\|^2\lambda^{10}_{max}(P_{max})C^4\|BR^{-1}_{min}B^{\mathsf{T}}\|^2(1+\|BR^{-1}_{min}B^{\mathsf{T}}\|^2)(\gamma^W+\eta^W)^2}{\lambda^{2}_{min}(R_{min})\lambda^{4}_{min}(Q_{min})(1-\eta)^2}\\
    &\geq \frac{4\|\bar{x}_{0}\|^2D\|A\|^2\|B\|^2\lambda^{10}_{max}(Q_{max})C^4\|B\bar{R}^{-1}_{min}B^{\mathsf{T}}\|^2\gamma^{2W}}{\lambda^{2}_{min}(\bar{R}_{min})\lambda^{4}_{min}(\bar{Q}_{min})},
\end{align*}
and
\begin{align*}
    F
    &\leq \frac{10D\gamma^{2W}\|\bar{x}_{0}\|^2C^4C_{K}^2}{3}\bigg[ (\alpha_{1}+\alpha_{2})\frac{\gamma^2}{(1-\gamma)^2}\frac{1}{1-\eta^2}+\frac{\eta^2}{1-\eta^2} \\
    &\quad + \frac{10C_{f}^2}{3q^2}\bigg((\frac{(\gamma-1)q\eta}{(q-\eta\gamma)(q-\eta)})^2\frac{1}{1-q^2} + \frac{(\eta\gamma)^2}{(q-\eta\gamma)^2(1-\eta^2\gamma^2)} + \frac{\eta^2}{(q-\eta)^2(1-\eta^2)}\bigg)\bigg]\\
    &\leq \frac{10D\gamma^{2W}\|\bar{x}_{0}\|^2C^4C_{K}^2}{3} \bigg[(1+\frac{\alpha_{1}+\alpha_{2}}{(1-\gamma)^2})(\frac{1}{1-\eta^2})+\frac{10C_{f}^{2}}{q^{2}(q-\eta\gamma)^{2}(q-\eta)^{2}(1-\eta^{2})(1-\eta^{2}\gamma^{2})(1-q^2)}\bigg].
\end{align*}
Thus, if
\begin{align*}
    \lambda_{max}^{10}(Q_{max}) \geq \frac{5\bigg[(1+\frac{\alpha_{1}+\alpha_{2}}{(1-\gamma)^2})(\frac{1}{1-\eta^2})+\frac{10C_{f}^{2}}{q^{2}(q-\eta\gamma)^{2}(q-\eta)^{2}(1-\eta^{2})(1-\eta^{2}\gamma^{2})(1-q^2)}\bigg]}{6(C_{K}^{2}\lambda^{2}_{min}(\bar{R}_{min})\lambda^{4}_{min}(\bar{Q}_{min}))^{-1}\|A\|^2\|B\|^2\|B\bar{R}^{-1}_{min}B^{\mathsf{T}}\|^2},
\end{align*}
then it follows that $F\leq F'$. 
\hfill $\blacksquare$
\section{Proof of Theorem~\ref{theorem:disturbanceResult}}\label{sec:proof_disturbance}

Lemma \ref{thm:lemma5} holds regardless of the presence or the absence of the disturbances. Thus,
\begin{equation}\label{eq:expectedRegret}
    \begin{split}
        \text{ExpectedRegret}_{T}(\{u_{t}\}_{t=0}^{T-2}) &= \E(\sum_{t=0}^{T-1}(u_{t}-\bar{u}_{t})^{\mathsf{T}}(R_{t}+B^{\mathsf{T}}P_{t+1}^{*}B)(u_{t}-\bar{u}_{t}))\\
    &\leq D\sum_{t=0}^{T-1}\E((u_{t}-\bar{u}_{t})^{\mathsf{T}}(u_{t}-\bar{u}_{t}))= D\sum_{t=0}^{T-1}\E(\|u_{t}-\bar{u}_{t}\|^{2}).
    \end{split}
\end{equation}
The state variable $x_{n\mid q}$ can be expressed as
\begin{equation}\label{eq:disStateXNQ}
    x_{n\mid q} = \prod_{j=0}^{n-1}(A+BK_{j\mid q})\bar{x}_{0} + \sum_{r=0}^{n-1}\bigg(\prod_{j=r+1}^{n-1}(A+BK_{j\mid q})\bigg)w_{r}.
\end{equation}
With $\theta_{i\mid p,q} = x_{i\mid p}-x_{i\mid q}$, we have that
\begin{equation}\label{eq:disStateNPQ}
    \begin{split}
        \theta_{i+1\mid p,q} &= (A+BK_{i\mid p})\theta_{i\mid p,q} + B(K_{i\mid p}-K_{i\mid q})x_{i\mid q}\\
    &= \sum_{n=0}^{i}\bigg(\prod_{m=n+1}^{i} (A+BK_{m\mid p})\bigg)B(K_{n\mid p}-K_{n\mid q})x_{n\mid q}.
    \end{split}
\end{equation}
Next, we state two lemmas to help us bounding the expected regret. 
\begin{lemma}\label{thm:disBound1}
    Consider the system in \eqref{eq:linsys} with initial condition $\bar{x}_{0} = 0$. Suppose that for any $t>0$, $\E(w_{t})=0$, $\E(w_{t}w_{t}^{\mathsf{T}})=W_{d}$ for a $W_{d} \in \mathbb{S}_{+}^{n}$. At any time $t$, disturbance sequence $\{w_{k}\}_{k=0}^{t}$ is available to the control policy \eqref{eq:policy}. Let $\{u_{t}\}_{t=0}^{T-2}$ denote the sequence that generate by control policy \eqref{eq:policy} subject to \eqref{eq:linsys}. There exist a positive scalar $C_{R2}$, the expected regret satisfies
    \begin{align}\label{eq:disExpReg}
        \text{ExpectedRegret}_{T}(\{u_{t}\}_{t=0}^{T-2}) \leq C_{R2}T\gamma^{2W}.
    \end{align}
\end{lemma}
\begin{proof}
    Substituting \eqref{eq:disStateXNQ} into \eqref{eq:disStateNPQ} with $\bar{x}_{0} = 0$ and $\theta_{i\mid p,q} := x_{i\mid p}-x_{i\mid q}$, we have that
    \begin{equation}\label{eq:expStateDiff1}
    \begin{split}
        \E(\|x_{t\mid t+W} - x_{t}^{*}\|^2) &= \E(\|\theta_{t\mid t+W,T-1}\|^2)\\
        &\leq \E(\|\sum_{n=0}^{t-1}\sum_{r=0}^{n-1}\bigg(\prod_{m=n+1}^{t-1}(A+BK_{m\mid t+W})\bigg)B(K_{n\mid t+W}-K_{n}^{*})\\
        &\qquad \bigg(\prod_{j=r+1}^{n-1}(A+BK_{j}^{*})\bigg)w_{r}\|^2)\\
        &\leq \frac{(C^2C_{K})^2\gamma^{2W}\eta^{2t}\gamma^{2t}}{\eta^{2}}\sum_{n_1=0}^{t-1}\sum_{n_2=0}^{t-1}\sum_{r_1=0}^{n_1-1}\sum_{r_2=0}^{n_2-1} \frac{\E(w_{r_1}w_{r_2}^{\mathsf{T}})}{\gamma^{n_1+n_2}\eta^{r_1+r_2}}
        := \kappa_{w\theta}
    \end{split}
    \end{equation}
    where $\kappa_{w\theta}=\gamma^{2W}(C_{\kappa_{w\theta}} + L_{\kappa_{w\theta}}(\eta^t,\gamma^t,\eta^{2t},\gamma^{2t}))$, $C_{\kappa_{w\theta}}$ is a non-negative scalar and $L_{\kappa_{w\theta}}(\eta^t,\gamma^t,\eta^{2t},\gamma^{2t})$ is a linear combination of $\eta^t,\gamma^t,\eta^{2t}$ and $\gamma^{2t}$.
    Following the same steps as in the proof of Lemma \ref{thm:lemma4} in Appendix \ref{section:lemma4_proof}, we have
    \begin{equation}\label{eq:expStateDiff2}
    \begin{split}
        \E(\|x_{t} - x_{t}^{*}\|^2) &\leq \E(\|\sum_{i=1}^{t-1}\sum_{n=0}^{i-1}\sum_{r=0}^{n-1}(A+BK)^{t-i}\bigg(\prod_{m=n+1}^{i-1} (A+BK_{m\mid i-1+W})\bigg)\\
        &\qquad B(K_{n\mid i-1+W}-K_{n}^{*})\bigg(\prod_{j=r+1}^{n-1}(A+BK_{j}^{*})\bigg)w_{r}\|^2)\\
        &\leq \underbrace{\kappa_{w\theta} + 2(C^2C_{K})^2\eta^{2t}\sum_{i_1=1}^{t-1}\sum_{i_2=1}^{t-1}\sum_{n_1=0}^{t-1}\sum_{n_2=0}^{t-1}\sum_{r_1=0}^{n_1-1}\sum_{r_2=0}^{n_2-1} \frac{\E(w_{r_1}w_{r_2}^{\mathsf{T}}))\gamma^{i_1+i_2-n_1-n_2}}{\eta^{r_1+r_2}}}_{\kappa_{wx}:= }
    \end{split}
    \end{equation}
    and
    \begin{equation}\label{eq:expStateOptimal}
    \begin{split}
        \E(\|x_{t}^{*}\|^2) &\leq \E(\|\sum_{r=0}^{t-1}\bigg(\prod_{j=r+1}^{t-1}(A+BK_{j}^{*})\bigg)w_{r}\|^2)\\
        &\leq C^4\sum_{r_1=0}^{t-1}\sum_{r_2=0}^{t-1} \E(w_{r_1}w_{r_2}^{\mathsf{T}})\eta^{2t-r_1-1}\eta^{2t-r_2-1} := \kappa_{wx^*}
    \end{split}
    \end{equation}
    where again $\kappa_{wx}$ and $\kappa_{wx^*}$ are of the form $\gamma^{2W}(C_{\kappa_{wx}} + L_{\kappa_{wx}}(\eta^t,\gamma^t,\eta^{2t},\gamma^{2t}))$ and $C_{\kappa_{wx^*}} + L_{\kappa_{wx^*}}(\eta^t,\gamma^t,\eta^{2t},\gamma^{2t})$, respectively, where $C_{\kappa_{wx}}$, $C_{\kappa_{wx^*}}$ are constants, $L_{\kappa_{wx}}(\eta^t,\gamma^t,\eta^{2t},\gamma^{2t})$ and $L_{\kappa_{wx^*}}(\eta^t,\gamma^t,\eta^{2t},\gamma^{2t})$ are linear combinations of $\eta^t,\gamma^t,\eta^{2t}$ and $\gamma^{2t}$.
    
    
    Recall $\alpha_1$, $\alpha_2$ and $D$ defined in Theorem \ref{theorem:mainResult}. Similar to the proof of Theorem \ref{theorem:mainResult}, there exist a positive scalar $C_{R2}$ that the expected regret satisfies
    \begin{equation}\label{eq:expRegIneq}
        \begin{split}
            \text{ExpectedRegret}_{T}(\{u_{t}\}_{t=0}^{T-2}) &\leq \frac{10D}{3}\sum_{t=0}^{T-1}(\alpha_{1}\E(\|x_{t\mid t+W} - x_{t}^{*}\|^2)+\alpha_{2}\E(\|x_{t} - x_{t}^{*}\|^2)\\
            &\quad +C^{2}_{K}\gamma^{2W}\E(\|x_{t}^{*}\|^2))\\
            &\leq \frac{10D}{3}\sum_{t=0}^{T-1}(\alpha_{1}\kappa_{w\theta} + \alpha_{2}\kappa_{wx} + C_{K}^{2}\gamma^{2W}\kappa_{wx^{*}})\\
            &\leq C_{R2}T\gamma^{2W}.
        \end{split}
    \end{equation}
    where the last inequality holds by substituting \eqref{eq:expStateDiff1}, \eqref{eq:expStateDiff2}, and \eqref{eq:expStateOptimal} into \eqref{eq:expectedRegret}.


\end{proof}
\begin{lemma}\label{thm:disBound2}
    For any $\bar{x}_{0}\in R^{n}$, consider the system define in \eqref{eq:linsys} has initial condition of $\bar{x}_{0}$. Suppose that for any $t>0$, $\E(w_{t})=0$, $\E(w_{t}w_{t}^{\mathsf{T}})=W_{d}$ for a $W_{d} \in \mathbb{S}_{+}^{n}$. At any time $t$, disturbance sequence $\{w_{k}\}_{k=0}^{t}$ is available to the decision maker. Let $\{u_{t}\}_{t=0}^{T-2}$ denote the sequence that generate by control policy \eqref{eq:policy} under the constrain of \eqref{eq:linsys}. We have
    \begin{align*}
        \text{ExpectedRegret}_{T}(\{u_{t}\}_{t=0}^{T-2}) \leq \text{RHS of inequality \eqref{eq:regUpperbound}} + \text{RHS of equality \eqref{eq:disExpReg}}.
    \end{align*}
\end{lemma}
\begin{proof}
    Let $\kappa_{\theta}$, $\kappa_{\omega}$, $\kappa_{x}$, $\kappa_{x^*}$, and $F_{exp}$ denote the RHS of \eqref{eq:estimateStateDifferenceBound}, \eqref{eq:realEstStateBound}, \eqref{eq:expStateDiff2}, \eqref{eq:stateUpperbound3} and \eqref{eq:disExpReg}, respectively. 
    Note that
    \begin{equation}\label{eq:disStateDiff1}
    \begin{split}
        \E(\|x_{t\mid t+W} - x_{t}^{*}\|^2) &= \E(\|\theta_{t\mid t+W,T-1}\|^2)\\
        &\leq \E(\|\sum_{n=0}^{t-1}\bigg(\prod_{m=n+1}^{t-1}(A+BK_{m\mid t+W})\bigg)B(K_{n\mid t+W}-K_{n}^{*})x_{n}^{*}\|^2)\\
        &\leq \E(\|\sum_{n=0}^{t-1}\bigg(\prod_{m=n+1}^{t-1}(A+BK_{m\mid t+W})\bigg)B(K_{n\mid t+W}-K_{n}^{*})\\
        &\quad\bigg[\prod_{j=0}^{n-1}(A+BK_{j}^{*})\bar{x}_{0} + \sum_{r=0}^{n-1}\bigg(\prod_{j=r+1}^{n-1}(A+BK_{j}^{*})\bigg)w_{r}\bigg]\|^2)\\
        &\leq \|\sum_{n=0}^{t-1}\bigg(\prod_{m=n+1}^{t-1}(A+BK_{m\mid t+W})\bigg)B(K_{n\mid t+W}-K_{n}^{*})\prod_{j=0}^{n-1}(A+BK_{j}^{*})\bar{x}_{0}\|^2 \\
        &\quad + \E(\|\sum_{n=0}^{t-1}\sum_{r=0}^{n-1}\bigg(\prod_{m=n+1}^{t-1}(A+BK_{m\mid t+W})\bigg)\\
        &\qquad B(K_{n\mid t+W}-K_{n}^{*})\bigg(\prod_{j=r+1}^{n-1}(A+BK_{j}^{*})\bigg)w_{r}\|^2)\\
        &\leq \kappa_{\theta}+\kappa_{w\theta}.
    \end{split}
    \end{equation}
    Similarly,
    \begin{equation}\label{eq:disStateDiff2}
    \begin{split}
        \E(\|x_{t}-x_{t}^{*}\|^2) &\leq 2\E(\|x_{t}-x_{t\mid t+W}\|^2+\|x_{t\mid t+W}-x_{t}^{*}\|^2)\\
        &\leq 2(\kappa_{w\theta}+\kappa_{\theta}) + 2\E(\|\sum_{i=1}^{t-1}(A+BK)^{t-i}\theta_{i\mid i-1+W,i+W}\|^2)\\
        &\leq 2\E(\sum_{i=1}^{t-1}\sum_{n=0}^{i-1}(A+BK)^{t-i}\bigg(\prod_{m=n+1}^{i-1} (A+BK_{m\mid i-1+W})\bigg)\\
        &\qquad B(K_{n\mid i-1+W}-K_{n}^{*})x_{n}^{*}) + 2(\kappa_{w\theta}+\kappa_{\theta})\\
        &\leq 2\E(\|\sum_{i=1}^{t-1}\sum_{n=0}^{i-1}(A+BK)^{t-i}\bigg(\prod_{m=n+1}^{i-1} (A+BK_{m\mid i-1+W})\bigg)B(K_{n\mid i-1+W}-K_{n}^{*})\\
        &\qquad \prod_{j=0}^{n-1}(A+BK_{j}^{*})\bar{x}_{0} + \sum_{r=0}^{n-1}\bigg(\prod_{j=r+1}^{n-1}(A+BK_{j}^{*})\bigg)w_{r}\|^2)+ 2(\kappa_{w\theta}+\kappa_{\theta})\\
        &\leq 2(\kappa_{w\theta}+\kappa_{\theta}) + 2(\kappa_{\omega}\\
        &\quad + \E(\|\sum_{i=1}^{t-1}\sum_{n=0}^{i-1}\sum_{r=0}^{n-1}(A+BK)^{t-i}\bigg(\prod_{m=n+1}^{i-1} (A+BK_{m\mid i-1+W})\bigg)\\
        &\quad B(K_{n\mid i-1+W}-K_{n}^{*})\bigg(\prod_{j=r+1}^{n-1}(A+BK_{j}^{*})\bigg)w_{r}\|^2)\\
        &\leq 2(\kappa_{w\theta}+\kappa_{\theta}) + 2(\kappa_{\omega} \\
        &\quad + (C^2C_{K})^2\eta^{2t}\sum_{i_1=1}^{t-1}\sum_{i_2=1}^{t-1}\sum_{n_1=0}^{t-1}\sum_{n_2=0}^{t-1}\sum_{r_1=0}^{n_1-1}\sum_{r_2=0}^{n_2-1} \frac{\E(w_{r_1}w_{r_2}^{\mathsf{T}}))\gamma^{i_1+i_2-n_1-n_2}}{\eta^{r_1+r_2}})\\
        &\leq  \kappa_{x} + \kappa_{wx},
    \end{split}
    \end{equation}
    the second step to the third step of inequality is by substituting \eqref{eq:disStateXNQ} in \eqref{eq:disStateNPQ}.
    Moreover,
    \begin{equation}\label{eq:disStateDiff3}
    \begin{split}
        \E(\|x_{t}^{*}\|^2) &\leq \E(\|\prod_{j=0}^{t-1}(A+BK_{j}^{*})\bar{x}_{0} + \sum_{r=0}^{t-1}\bigg(\prod_{j=r+1}^{t-1}(A+BK_{j}^{*})\bigg)w_{r} \|^2)\\
        &\leq \|\prod_{j=0}^{t-1}(A+BK_{j}^{*})\bar{x}_{0}\|^2 + \E(\|\sum_{r=0}^{t-1}\bigg(\prod_{j=r+1}^{t-1}(A+BK_{j}^{*})\bigg)w_{r}\|^2)\\
        &\leq  \kappa_{x^*}+\kappa_{wx^*}.
    \end{split}
    \end{equation}
    Substituting \eqref{eq:disStateDiff1}, \eqref{eq:disStateDiff2} and \eqref{eq:disStateDiff3} into \eqref{eq:expectedRegret}, we have
    \begin{align*}
        \text{ExpectedRegret}_{T}(\{u_{t}\}_{t=0}^{T-2}) &\leq \frac{10D}{3}\sum_{t=0}^{T-1}(\alpha_{1}\E(\|x_{t\mid t+W} - x_{t}^{*}\|^2)\\
        &+\alpha_{2}\E(\|x_{t} - x_{t}^{*}\|^2)+C^{2}_{K}\gamma^{2W}\E(\|x_{t}^{*}\|^2))\\
        &\leq \sum_{t=0}^{T-1}  \frac{10D}{3}\alpha_1(\kappa_{1}+\kappa_{\theta}) + \alpha_2(\kappa_{x}+\kappa_{wx}) \\
        &\quad + C_{K}^{2}\gamma^{2W}(\kappa_{wx^*}+\kappa_{x^*})\\
        &\leq F + F_{exp}.
    \end{align*}
\end{proof}
Now, we are ready to proof Theorem \ref{theorem:disturbanceResult}.
\begin{proof}
Based on the expression of $F$ and $F_{exp}$ from Theorem \ref{theorem:mainResult} and Lemma \ref{thm:disBound1}, there exist positive scalars $C_{R1}$ and $C_{R2}$, such that $F \leq \gamma^{2W}C_{R1}$ and $F_{exp} \leq \gamma^{2W}TC_{R2}$ for $T \geq 2$ and $0 \leq W \leq T-2$. By Lemma \ref{thm:disBound2}, for the control sequence $\{u_{t}\}_{t=0}^{T-2}$ generated by control policy \eqref{eq:policy}, we have
\begin{align*}
     \text{ExpectedRegret}_{T}(\{u_{t}\}_{t=0}^{T-2}) &\leq F+F_{exp}\\
    &\leq \gamma^{2W}(C_{R1}+C_{R2}T).
\end{align*}
Let $C_{ER} = C_{R1} + C_{R2}$, we have
\begin{align*}
    \text{ExpectedRegret}_{T}(\{u_{t}\}_{t=0}^{T-2}) \leq \gamma^{2W}C_{ER}T.
\end{align*}
\end{proof}
\end{document}